\documentclass[oneside,english,reqno]{amsart}
\usepackage{algorithm}
\usepackage{algorithmic}
\usepackage{multirow}
\usepackage{subfigure}
\usepackage{enumerate}

\usepackage[T1]{fontenc}
\usepackage[latin9]{inputenc}
\usepackage{babel}
\usepackage{amsthm}
\usepackage{amssymb}
\usepackage{graphicx}
\usepackage[unicode=true,pdfusetitle,
bookmarks=true,bookmarksnumbered=false,bookmarksopen=false,
breaklinks=false,pdfborder={0 0 1},backref=false,colorlinks=false]
{hyperref}

\makeatletter
\numberwithin{equation}{section}
\numberwithin{figure}{section}

\makeatother
\makeatletter
\renewcommand{\ALG@name}{Iterative Scheme}
\makeatother

\newtheorem{proposition}{Proposition}
\newtheorem{theorem}{Theorem}
\newtheorem{lemma}{Lemma}
\newtheorem{remark}{Remark}
\newtheorem{corollary}{Corollary}
\newtheorem{definition}{Definition}
\newtheorem{example}{Example}
\begin{document}

\title{Closed-form expressions for projectors onto polyhedral sets in Hilbert spaces}

\author{
	Krzysztof E. Rutkowski$^1$ 
}
\thanks{$^1$ Warsaw University of Technology, 		 \href{mailto:k.rutkowski@mini.pw.edu.pl}{k.rutkowski@mini.pw.edu.pl} .}
\maketitle
\tableofcontents{}

\begin{abstract}
	We provide formulas for projectors onto a polyhedral set, i.e. the intersection of a finite number of halfspaces. To this aim we formulate the problem of finding the projection as a convex optimization problem and we solve explicitly sufficient and necessary optimality conditions. This approach has already been successfully applied in
	deriving formulas for projection onto the intersection of two halfspaces. 
	We also discuss possible generalizations to Banach spaces.
\end{abstract}

\section{Introduction}
%\mbox{\,}\indent
Projections onto convex closed sets play an important role in constructions of algorithms for solving optimization problems 
(see \cite{complementatity_web} for nonlinear complementarity problems and variational inequalities).
In general, projections onto intersections of convex sets
are obtained as limits of iterative processes, see e.g.  %Pang, von Neumann, 
\cite{a_new_projection_method, Accelerating_the_Convergence, swiss_army_knives,  cegielski2012iterative,  GUBIN19671, projection_methods_im_conic,  an_extended_projective_formula}. %For most recent publications see also \cite{a_new_projection_method, swiss_army_knives, cegielski2012iterative, projection_methods_im_conic,  an_extended_projective_formula}

The idea of finding a closed-form expression for projector onto a linear subspace by solving explicitly the corresponding optimization problem goes back to Pshenichnyj \cite[Theorem 1.19]{Pshenichnyj}. This idea has also been used by Bauschke and Combettes in \cite[Proposition 28.19, Proposition 28.20]{Bauschke_2011_convex_analysis} to provide explicit formulas for the projection onto the intersection of two halfspaces. 
The obtained formulas were at the core of the algorithm approximating the 
Kuhn-Tucker set for the  pair of dual monotone inclusions as proposed in \cite{Alotaibi_2015_best}.
Recently, a finite algorithm for projection onto an isotone projection cone was given in \cite{nemeth_how_to_project}. % N{\'e}meth -  finite
This algorithm allows one to improve considerably the performance of a class of algorithms for solving complementarity problems \cite{iterative_methods_for_nonlinear}.

In this paper we provide a closed-form expression for the projector onto polyhedral sets in Hilbert spaces. The results of the present paper are applied in constructing inertial algorithms for approximation of the Kuhn-Tucker set for monotone inclusions \cite{inertial_proximal}.

The starting point of our considerations is Theorem 6.41 of \cite{deutsch2001best} which provides the Kuhn-Tucker conditions for the convex optimization problem related to projections. Analogous approach was presented in \cite{on_an_exact} to solve explicitly an optimization problems with second order cone constraints. It is essential in our approach that:
(1) in Hilbert space we have a simple formula for the derivative of the norm, (2) the number of halfspaces is finite. An analogous approach %to finding projection onto finite number of halfspaces 
in Banach spaces depends strongly on differentiability properties of the norm. 

Let $H$ be a real Hilbert space equipped with a scalar product $\langle \cdot \ |\ \cdot \rangle :\ H\times H \rightarrow \mathbb{R}$ and the associated norm $\|\cdot \|$. For any closed convex subset $K\subset H$, let $P_K(x)$ denote the projection of
$x\in H$ onto $K$. Let $C_i:=\{ h \in H \ |\ \langle h \ |\ u_i \rangle \leq \eta_i \}$, $u_i \in H\backslash\{0\}$, $\eta_i\in \mathbb{R}$,  for $i=1,\dots,n$ be a finite family of halfspaces.
Halfspaces are clearly convex sets, and, by the Riesz representation theorem, halfspaces are also closed subsets of the Hilbert space $H$.

Let $C\subset H$ be defined as
\begin{equation*}
	C:=\bigcap_{i=1}^n C_i.
\end{equation*}

We derive closed-form expressions for the projector $P_C(x)$ of an element $x\in H$ onto $C$ when $C$ is nonempty. Our framework takes into account all possible relationships between vectors $u_i$ for $i=1,\dots,n$, e.g. we do not assume that $u_{i}$, $i=1,\dots,n$ are linearly independent. 	
In Banach spaces this approach does not provide, in general, explicit formulas for projections. In some particular cases we give verifiable criteria to check whether a given $\bar{x}$ is a projection. 

The organization of the paper is as follows. In section \ref{section:Projection} 
%Section \ref{section:Projection} 
we provide a refinement of the existing theorem on formulas for the projection onto polyhedral sets (Proposition \ref{prop:system_coeff}).
It is achieved by the analysis of optimality conditions for the optimization problem related to finding the projection $P_C(x)$.
In section \ref{section:results}
%Section \ref{section:results}
we use Proposition \ref{prop:system_coeff} to provide explicit formulas for projection $P_C(x)$. This is the content of Theorem \ref{prop:projection_conditions} which is our main result. 
In section \ref{section:latticial}
%Section \ref{section:latticial}
we compare our approach with the already existing approaches to provide explicit formulas for projections, in particular we compare Theorem \ref{prop:projection_conditions} of section \ref{section:results}
%Section \ref{section:results}
with Theorem 2 of \cite{nemeth_how_to_project}.
%   In section \ref{section:examples}
%  %Section \ref{section:examples}
%   we apply the obtained results to the case of $n=3$.
In section \ref{section:Banach} we investigate projections onto $C$ in Banach spaces and we provide some criteria to verify whether $\bar{x}$ is a projection of $x$.

\textbf{Notation.} Let $n$ be a strictly positive integer number. We reserve the symbol $N$ to the set defined as $N:=\{1,2,\dots,n\}$. Symbols $H,B$ denote Hilbert and Banach space, respectively.  
A function $g:\ H\rightarrow (-\infty,\infty]$ is proper if it is not equal to $+\infty$ on the whole space.
%The symbol $\text{lev}_{\leq 0} g$ denote the set of $x\in H$ such that $g(x)\leq 0$ and the symbol $\text{lev}_{< 0} g$ denotes the set of $x\in H$ such that $g(x)< 0$.
When $G$ is a matrix of dimensions $m\times k$ and $I\subset \{1,\dots,m\}$, $J\subset\{1,\dots,k\}$, $I,J\neq \emptyset$ the symbol $G_{I,J}$ denotes the submatrix of $G$ composed by rows indexed by $I$ and columns indexed by $J$ only. For any Gateaux differentiable function $f:\, B\rightarrow \mathbb{R}$, $f^\prime (x)$ denote the Gateaux derivative of $f$ at $x$ and $f^\prime(x,d)$ denotes the directional derivative of $f$ at $x$ in direction $d$. For any $i,j \in \mathbb{N}$ the symbol $\delta_{i,j}$ denotes the Kronecker delta.

\section{Projections}\label{section:Projection}
%\mbox{\,}\indent 

Let $x \in H$, $C=\bigcap_{i\in N} C_i$, where $C_i=\{h \in H \ |\ \langle h \ |\ u_i \rangle \leq \eta_i   \}$, $i \in N$. The optimization problem
\begin{equation}\label{problem:optimization_distance}
	%\min_{\substack{h\in H\\ \langle h \ |\ u_i\rangle\le \eta_i}} 
	\min_{h \in C}\ \frac{1}{2}\|h-x\|^2
\end{equation}
is equivalent to finding the projection of $x$ onto $C$. 
%In this case the Slater condition reduces to the nonemptiness of the feasible set 
%(%see Section 5.2.3 of \cite{convex_optimization}
%see Theorem 11.3 of \cite{Lectures_on_Mathematical_Theory}).
This is a quadratic programming problem with linear inequality constraints on $H$  (see \cite{perturbation_analysis_of_optimization_problems, quadratic_programming_in_real_hilbert_spaces} and \cite{formulations_of_support_vector_machines} for applications).

To this problem we can apply the optimality conditions for general convex optimization problem of the form
\begin{align}
	\begin{aligned}\label{prob:minimaztion_functional}
		\min_{x\in H}  F_0(x)\\
		F_i(x)\leq 0, i\in N,
	\end{aligned}
\end{align}
where $F_0,\ F_i:\ H\rightarrow \mathbb{R}$, $i\in N$ are functionals on  $H$. In the sequel we use the following form of the Kuhn-Tucker conditions for the problem \eqref{prob:minimaztion_functional}\footnote{Originally generalizations of the Kuhn-Tucker necessary optimality conditions to infinite dimensional spaces were given in \cite{studies_in_linear_and_nonlinear}, \cite{convex_programming_in_normalized}, \cite{nonlinear_programming_in_Banach}.  }.

\begin{proposition}\label{prop:optimization_problem1} \cite[Theorem 11.3]{Lectures_on_Mathematical_Theory} (see also \cite[Proposition 3.118]{perturbation_analysis_of_optimization_problems})
	Let $F_0$, $F_i:\ H\rightarrow \mathbb{R}$, $i\in N$ be convex and
	continuously differentiable on $H$ and 
	%the constraints are linear:
	\begin{equation*}
		F_i(x):=\langle x \ |\ u_i \rangle - \eta_i,\quad 0\neq u_i \in H,\quad \eta_i \in \mathbb{R},\quad i \in N.
	\end{equation*}
	Sufficient and necessary conditions for minimum at $\bar{x}\in H$ are:
	\begin{align}
		\begin{aligned}\label{system:propsolution}
			\nu_i\geq 0,\quad F_i(\bar{x})\leq 0, \quad \nu_i F_i(\bar{x})=0,\quad i \in N,\\
			F_0^\prime (\bar{x}) + \nu_1 F_1^\prime (\bar{x}) + \dots + \nu_n F_n^\prime (\bar{x}) = 0.
		\end{aligned}
	\end{align}
\end{proposition}

Applying  Proposition \ref{prop:optimization_problem1} to functions  $F_0(\cdot)=\frac{1}{2} \| \cdot - x \|^2$ and $F_i(\cdot)=\langle \cdot \ |\ u_i \rangle - \eta_i $, $i\in N$, we obtain the following theorem due to Deutsch \cite{deutsch2001best}.

%The following theorem has been proved in \cite{deutsch2001best}. 

\begin{theorem}\cite[Theorem 6.41]{deutsch2001best}\label{theorem:Deutsch}
	Let $u_i\in H $, $\eta_i \in \mathbb{R}$, $i\in 1,\dots,n$ and $C=\bigcap_{i\in N} C_i\neq \emptyset$. If $x \in H$ then
	\begin{displaymath} %\label{theorem:Deutsch:projection}
		P_C(x)=x-\sum_{i=1}^n \nu_i u_i,
	\end{displaymath} 
	for any set of scalars $\nu_i$ that satisfy the following three conditions:
	\begin{equation}\label{theorem:Deutsch:cond1}
		\nu_i \geq 0, \quad \text{for}\ i=1,\dots,n,
	\end{equation}
	\begin{equation}\label{theorem:Deutsch:cond2}
		\langle x, u_i \rangle - \eta_i -\sum_{j=1}^n \nu_i \langle u_j \ |\  u_i \rangle \leq 0 \quad \text{for}\ i=1,\dots,n,
	\end{equation}
	\begin{equation}\label{theorem:Deutsch:cond3}
		\nu_i(\langle x, u_i \rangle - \eta_i -\sum_{j=1}^n \nu_i \langle u_j \ |\  u_i \rangle) = 0 \quad \text{for}\ i=1,\dots,n.
	\end{equation}
	Consequently, if $x \in H$ and $\bar{x}\in C$, then $\bar{x}=P_C(x)$ if and only if
	\begin{displaymath}
		\bar{x}=x-\sum_{i\in I(\bar{x})} \nu_i u_i,\quad  \text{for some}\ \nu_i\geq 0,
	\end{displaymath}
	where $I(\bar{x}):=\{i \in N \ |\ \langle \bar{x} \ |\ u_i \rangle = \eta_i  \}$.
\end{theorem}

This fact is proved in \cite{deutsch2001best} as an immediate consequence of the projection theorem in Hilbert spaces and representations of dual cones. 
Note that conditions \eqref{theorem:Deutsch:cond1}, \eqref{theorem:Deutsch:cond2}, \eqref{theorem:Deutsch:cond3} of Theorem \ref{theorem:Deutsch} are conditions \eqref{system:propsolution} of Proposition \ref{prop:optimization_problem1} when applied to problem \eqref{prob:minimaztion_functional}.

%To make this theorem operational the scalars $\nu_i$ are chosen according to \eqref{theorem:Deutsch:cond1}, \eqref{theorem:Deutsch:cond2}, \eqref{theorem:Deutsch:cond3}. 

Let $\{u_i\}\in H$, $\eta_i\in\mathbb{R}$, $i\in N$ and let 
\begin{displaymath}
	G:=\left[
	\begin{array}{cccc}
		\|u_1\|^2 & \langle u_1 \ |\ u_2\rangle & \cdots & \langle u_1 \ |\ u_n \rangle\\
		\langle u_2 \ |\ u_1 \rangle & \|u_2\|^2 & & \langle u_2 \ |\ u_n \rangle\\
		\vdots & & \ddots  & \vdots\\
		\langle u_n \ |\ u_1 \rangle & \langle u_n \ |\ u_2 \rangle & \cdots & \|u_n\|^2
	\end{array}\right] .
\end{displaymath}

The matrix $G$ is called the \textit{Gram matrix} and has the following well-known property: for any $I\subset N$  $\det G_{I,I}\geq 0$ and $\det G_{I,I}=0$  if and only if vectors $u_i$, $i\in I$ are linearly dependent.

In Proposition \ref{prop:system_coeff} we derive equivalent conditions on scalars $\nu_i$. Due to the form of  \eqref{theorem:Deutsch:cond1}, \eqref{theorem:Deutsch:cond2}, \eqref{theorem:Deutsch:cond3} 
%by analysing the convex optimization problem corresponding to finding projection onto the intersection of a finite number of halfspaces. 
these conditions can be expressed in terms of the existence of positive solutions of systems of linear equations.

% to satisfy \eqref{theorem:Deutsch:projection} 

\begin{proposition}\label{prop:system_coeff} Let $C=\bigcap_{i\in N} C_i\neq \emptyset$, where 
	$C_i=\{h \in H \ |\ \langle h \ |\ u_i \rangle \leq \eta_i   \}$, $u_i 
	\in H\backslash\{0\}$, $\eta_i\in\mathbb{R}$ for $i\in N$ and let $x\in H\backslash C$.
	% Let $I\subset N$, $I\neq \emptyset$ and $\det G_{I,I}\neq 0$. 
	The point $\bar{x}$ is a projection of $x$ onto $C$ if and only if there exists $I\subset N$, $I\neq \emptyset$ such that (feasibility conditions)
	\begin{equation}\label{projection:formula}
		\bar{x}=x-\sum\limits_{i\in I} \nu_i u_i \in C,
	\end{equation}
	where $\nu_i$, $i\in I$, solve the following system of linear equations (complementarity slackness conditions)
	\begin{equation}\label{system:solution}
		\forall i \in I\quad \left\{
		\begin{array}{l}
			\langle x \ |\ u_i \rangle - \eta_i= \sum_{k\in I} \nu_k \langle u_k  \ |\ u_i \rangle, \\
			\nu_i>0.
		\end{array} \right.
	\end{equation}
	Moreover, there always exists at least one $I$ for which: (1) $\det G_{I,I}>0$, (2) system \eqref{system:solution} is solvable, (3) formula \eqref{projection:formula} holds.
\end{proposition}

The main contribution of this proposition is condition \eqref{system:solution} which replaces conditions \eqref{theorem:Deutsch:cond1}, \eqref{theorem:Deutsch:cond2}, \eqref{theorem:Deutsch:cond3} of Theorem \ref{theorem:Deutsch} and reduces the question of finding the projection onto $C$ to solving a consistent system of linear equations.

The proof of Proposition \ref{prop:system_coeff} is based on the following technical lemma.
\begin{lemma}\label{lemma:tech}(see, e.g. \cite[Lemma 6.33]{deutsch2001best})
	Let $u_i \in H$, $u_i \neq 0$, and $\tilde{\nu}_i\geq 0$ for $i\in N$, not all equal zero, and $w:=\sum_{i\in N} \tilde{\nu_i} u_i\neq 0$. There exist $I\subset N$ and $\nu_i>0$, $i\in I$ such that
	\begin{equation*}
		w=\sum_{i \in I} \nu_i u_i \quad \text{and} \quad \det G_{I,I}>0.
	\end{equation*} 
\end{lemma}

\begin{proof}[Proof  of Proposition \ref{prop:system_coeff} ]
	%	Let $F_0(\cdot)=\frac{1}{2} \| \cdot - x \|^2$ and $F_i(\cdot)=\langle \cdot \ |\ u_i \rangle -\eta_i$ for $i \in N$. We have
	%	\begin{align*}
	%		F_0^\prime(\cdot) &= \cdot - x,\\
	%		F_i^\prime(\cdot) &= u_i,\quad i\in N.
	%	\end{align*}
	%	Applying  Proposition \ref{prop:optimization_problem1} to
	%	$F_0$ and $(F_i)_{i \in N}$ we conclude that there exist $\tilde{\nu}_i$, $i\in N $, such that 
	By Theorem \ref{theorem:Deutsch}, there exists $\tilde{\nu}_i$, $i\in N $, such that 
	\begin{equation}\label{formula:system}
		\forall i \in N\quad \left\{
		\begin{array}{l}
			\langle x - \sum_{k\in N} \tilde{\nu}_k u_k \ |\ u_i \rangle - \eta_i  \leq 0, \\
			\tilde{\nu}_i (\langle x - \sum_{k\in N} \tilde{\nu}_k u_k \ |\ u_i \rangle - \eta_i)=0,\\
			\tilde{\nu}_i\geq 0
		\end{array} \right.
	\end{equation}	
	and $\bar{x}$ defined as 
	\begin{equation}\label{formula:projection1}
		\bar{x}:=x-\sum\limits_{i\in N} \tilde{\nu}_i u_i
	\end{equation}
	is the projection of $x$ onto $C$.
	%solves 
	%	\eqref{problem:optimization_distance}. Since solution of  \eqref{problem:optimization_distance} is the projection of $x$ onto $C$ we obtain $\bar{x}=P_C(x)$.
	Let $J:=\{ i\in N \ |\ \tilde{\nu}_i>0 \}$. Since $x\notin C$, from \eqref{formula:projection1} we deduce that $J\neq \emptyset$. We rewrite formula \eqref{formula:projection1} in the form
	\begin{equation}\label{formula:projection2}
		\bar{x}=x-\sum\limits_{i\in J} \tilde{\nu}_i u_i
	\end{equation}
	and by \eqref{formula:system} we have
	\begin{equation}\label{formula:system2}
		\forall i \in J\quad \left\{
		\begin{array}{l}
			\langle x - \sum_{k\in J} \tilde{\nu}_k u_k \ |\ u_i \rangle - \eta_i  = 0, \\
			\tilde{\nu}_i> 0.
		\end{array} \right.
	\end{equation}
	The system \eqref{formula:system2} is of the form
	\begin{equation}\label{formula:system3}
		G_{J,J}[\tilde{\nu}_i]_{i \in J}=[\langle x \ |\ u_i \rangle - \eta_i]_{i \in J}
	\end{equation}
	and, by \eqref{formula:projection2}, we know that \eqref{formula:system3} has a strictly positive solution $\tilde{\nu}_i>0$, $i \in J$. 
	If $\det G_{J,J}=0$, then, by Lemma \ref{lemma:tech}, there exists $I\subset J$ and $\nu_i>0$, $i\in I$, such that $\det G_{I,I}\neq 0$ and $\sum_{j\in J} \tilde{\nu}_j u_j=\sum_{i\in I}\nu_i u_i$.
	The index set $I$ satisfies the requirements given in the assertion of the proposition.
\end{proof}
%\begin{remark} %\ref{prop:system_coeff}
Let us note that the index set $I$ might be a one element set.
%\end{remark}

\section{Main results}\label{section:results}
%\mbox{\,}\indent
In this section we provide explicit formulas for solutions to  optimization problem \eqref{problem:optimization_distance}. This is the content of Theorem \ref{prop:projection_conditions} which is our main result.

Let %$N:=\{1,2\dots,n\}$, 
$I\subset N$ and $s_I(a):=\{ b \in I \ |\ b\leq a  \}$. We define
\begin{equation*}
	B_I^a:=\left\{\begin{array}{lcl}
		(-1)^{|s_I(a)|} & \text{if} & a \in I,  \\
		(-1)^{|I|+1} & \text{if} & a \notin I.
	\end{array}\right.
\end{equation*}

Let $w_i:=\langle x \ |\ u_i \rangle -\eta_i$, $i \in N$.
\begin{theorem} \label{prop:projection_conditions}
	Let $C=\bigcap\limits_{i=1}^n C_i\neq\emptyset$, where $C_i=\{h \in H \ |\ \langle h \ |\ u_i \rangle \leq \eta_i   \}$, $u_i\neq 0$, $\eta_i\in\mathbb{R}$, $i\in N$, $x\notin C$. Let $\text{rank}\ G=k$. Let $\emptyset \neq I\subset N$, $|I|\leq k$ be such that $\det G_{I,I}\neq 0$. Let
	\begin{equation}\label{coeff:plus}
		\nu_i:= \left\{
		\begin{array}{lcl}
			\sum_{j \in I} w_j B_I^{j} B_I^{i}
			%\delta_i^j 
			\det G_{I\backslash j , I\backslash i} & \text{if} & |I|>1,\\
			w_i & \text{if} & |I|=1
		\end{array}\right.
		\quad \text{for all} \quad i \in I
	\end{equation}
	and, whenever  $ I^\prime:=N\backslash I$ is nonempty, let
	\begin{equation}\label{coeff:notplus}
		\nu_{i^\prime}:= 
		\sum_{j \in I\cup \{i^\prime\}} w_j B_I^{j} B_I^{i^\prime}
		\det G_{I , (I \cup i^\prime)\backslash j  }
		\quad \text{for all} \quad i^\prime \in I^\prime.
	\end{equation}

	If $\nu_i>0$ for $i \in I$ and $\nu_{i^\prime}\leq 0$ for all $i^\prime\in I^\prime$, then 
	\begin{equation*}
		P_C(x)=x-\sum\limits_{i \in I} \frac{\nu_i}{\det G_{I,I}}u_{i}.
	\end{equation*}
	Moreover, among all the elements of the set $\Delta$ of all subsets $I\subset N$ there exists at least one $I\in \Delta$ for which: (1) $\det G_{I,I}\neq 0$, (2) the coefficients $\nu_i$, $i \in I$ given by \eqref{coeff:plus} are positive, (3) the coefficients $\nu_{i^\prime}$, $i^\prime \in I^{\prime}$ given by \eqref{coeff:notplus} are nonpositive.
\end{theorem}
\begin{proof}
	Let $I\subset N$, $I\neq \emptyset$, $\det G_{I,I}\neq 0$. Let $\nu_i$ be given by \eqref{coeff:plus} for $i \in I$ and let $\nu_{i^\prime}$ be given by \eqref{coeff:notplus} for $i^\prime \in I^\prime$. Assume that $\nu_i>0$ for $i\in I$ and $\nu_{i^\prime}\leq 0$ for $i^\prime \in I^\prime$. The elements  $\tilde{\nu_i}:=\frac{\nu_i}{\det G_{I,I}}$ for $i \in I$ solve the system
	\begin{equation}\label{system:nuplus}
		\forall i \in I\quad \left\{
		\begin{array}{l}
			\langle x \ |\ u_i \rangle-\eta_i= \sum_{k\in I} \tilde{\nu}_k \langle u_k  \ |\ u_i \rangle, \\
			\tilde{\nu}_i>0
		\end{array} \right. 
	\end{equation}
	or, in the matrix form,
	\begin{displaymath} %\label{system:nuplus2}
		G_{I,I}[\tilde{\nu}_i]_{i \in I}=[\langle x \ |\ u_i \rangle - \eta_i]_{i \in I},
	\end{displaymath}
	where the solution of this system satisfies $\tilde{\nu}_i>0$, $i \in I$.
	Thus, for all $i\in I$ we have $\tilde{\nu}_i( \langle x - \sum_{k \in I} \tilde{\nu}_k u_k\ |\  u_i \rangle-\eta_i)=0$ and consequently $\langle x-\sum_{k\in I } \tilde{\nu}_k u_k \ |\ u_i \rangle -\eta_i= 0  $ for $i \in I$.
	
	To prove that $\bar{x}=x-\sum_{i\in I} \nu_i u_i$ is the projection of $x$ onto $C$ it is enough to show that $\langle \bar{x} \ |\ u_{i^\prime} \rangle -\eta_{i^\prime} \leq 0$ for all $i^\prime \in I^\prime$. It is obvious in case  $I^\prime=\emptyset$, so suppose $I^\prime$ is nonempty. 
	\begin{enumerate} 
		\item Suppose $I=\{m\}$, where $m \in N$. Let $i^ \prime \in I^\prime $. Then
		\begin{align*}
			&\| u_m \|^2 (\langle x - \tilde{\nu}_m u_m \ |\ u_{i^\prime}  \rangle -\eta_{i^\prime}) = 	\| u_m \|^2 \langle x- \frac{\langle x \ |\ u_m \rangle-\eta_m}{\|u_m\|^2}        u_m    \ |\ u_{i^\prime}  \rangle-\|u_m\|^2\eta_i\\
			& = (\langle x \ |\ u_{i^\prime} \rangle -\eta_{i^\prime})\|u_m\|^2 - (\langle x \ |\ u_m \rangle - \eta_m) \langle u_m \ |\ u_{i^\prime} \rangle \\
			& = \sum_{j\in \{m,i^\prime\}} (\langle x \ |\ u_j \rangle - \eta_j)
			B_I^{j} B_I^{i^\prime}
			\det G_{I,(I\cup \{i^\prime \})\backslash \{j\}}\leq 0.
		\end{align*}
		Since $\|u_m\|^2>0$, $\langle \bar{x} \ |\ u_{i^\prime}  \rangle - \eta_{i^\prime} \leq 0 $ for all $i^\prime \in I^\prime$.
		\item Suppose $|I|\geq 2$. Let $i^ \prime \in I^\prime $. Then
		\begin{align*}
			&\det G_{I,I} ( \langle x - \sum_{k\in I}\tilde{\nu}_k u_k \ |\ u_{i^\prime}  \rangle -\eta_{i^\prime}	)	= (\langle x \ |\ u_{i^\prime} \rangle - \eta_{i^\prime})\det G_{I,I} - \sum_{k\in I} \nu_k  \langle u_k \ |\ u_{i^\prime} \rangle\\
			&=  \left(\langle x \ |\ u_{i^\prime} \rangle -\eta_{i^\prime} \right)\det G_{I,I} - \sum_{k\in I} \sum_{j \in I} (\langle x \ |\ u_j \rangle -\eta_j )B_I^j B_I^{k} \det G_{I\backslash \{j\} , I\backslash \{k\}}   \langle u_k \ |\ u_{i^\prime} \rangle\\ 
			&=  (\langle x \ |\ u_{i^\prime} \rangle -\eta_{i^\prime})\det G_{I,I} - \sum_{j \in I}( (\langle x \ |\ u_j \rangle-\eta_j)\sum_{k\in I}  B_I^j B_I^{k} \det G_{I\backslash \{j\} , I\backslash \{k\}}   \langle u_k \ |\ u_{i^\prime} \rangle)\\ 		
			& = (\langle x \ |\ u_{i^\prime}-\eta_{i^\prime}) \rangle  \det G_{I,I} -\sum_{j\in I} (\langle x \ |\ u_j \rangle - \eta_j) B_I^j B_{I\backslash j}^{i^\prime} \det G_{(I\cup \{i^\prime\})\backslash \{j\} ,I } \\
			& = (\langle x \ |\ u_{i^\prime} \rangle - \eta_{i^\prime}) \det G_{I,I} +\sum_{j\in I} (\langle x \ |\ u_j \rangle -\eta_j)B_I^j B_I^{i^\prime}  \det G_{I ,(I\cup \{i^\prime\})\backslash \{j\}  } \\
			& = \sum_{j \in I\cup \{i^\prime\}} (\langle x \ |\ u_j \rangle - \eta_j) B_I^j B_I^{i^\prime}  \det G_{I,(I\cup \{i^\prime\} )\backslash \{j\}}\leq 0.
		\end{align*}
		Since $\det G_{I,I}>0$, $\langle \bar{x} \ |\ u_{i^\prime}  \rangle - \eta_{i^\prime} \leq 0 $  for all $i^\prime \in I^\prime$. 
		
	\end{enumerate}
	The existence of $I\subset N$ such that $\nu_i>0$, $i\in I$ and $\nu_{i^\prime}\leq 0$, $i^\prime \in I^\prime$ is guaranteed by Proposition \ref{prop:system_coeff}.  	
\end{proof}

%\begin{remark}
It is easy to see from the proof of Theorem \ref{prop:projection_conditions} that $\tilde{\nu}_i:=\frac{\nu_i}{\det G_{I,I}}$, $i\in I$, with $\nu_i$ defined by \eqref{coeff:plus}, solve system \eqref{system:nuplus} (complementarity slackness conditions) which can be rewritten as
\begin{displaymath}
	G_{I,I}[\bar{\nu}_i]_{i \in I}=[\langle x \ |\ u_i \rangle - \eta_i]_{i \in I},
\end{displaymath}
whereas $(\nu_{i^\prime})_{i^\prime \in I^\prime}$ allows us to check whether the resulting $\bar{x}$ belongs to the set $C$ (feasibility conditions).
%\end{remark}

Theorem \ref{prop:projection_conditions} suggests the following finite algorithm for finding the projection $P_C(x)$ for $x\notin C$:
% % %
%	Step 1. For any subset $I\subset N$ such that $\det G_{I,I}\neq 0$ solve the linear system
%	\begin{equation}\label{step:1}
%	\langle x \ |\ u_i \rangle - \eta_i = \sum_{k\in I} \tilde{\nu}_k \langle u_k \ |\ u_i \rangle,\quad i \in I
%	\end{equation}
%	 with respect to $\tilde{\nu}_k$, $k\in I$.
%	Then, from the family of all subsets of $N$, select the subfamily $\Delta$ of such index sets $I$ for which system \eqref{step:1} possesses positive solutions $\tilde{\nu}_k$, $k\in I$ (the existence of $I$ follows from Theorem \ref{prop:projection_conditions}).
%	\par
%	Step 2. For each $I \in \Delta$ check whether solutions $\tilde{\nu}_k$, $k \in I$ given by \eqref{step:1} satisfy the inequalities
%	\begin{equation}\label{step:2}
%		\forall i^\prime \in I^\prime \quad \langle x-\sum_{k \in I} \tilde{\nu}_k u_k \ |\ u_{i^\prime} \rangle \leq 0. 
%	\end{equation}
% % %
Let $\Delta:=\{I_1,I_2,\dots,I_{2^n-1} \}$ be the collection of all nonempty subsets of $N$ and let $m=1$.\par
Step 1. Check, if $\det G_{I_m,I_m}\neq 0$. If not, let $m:=m+1$ and repeat Step 1.\par
Step 2. Solve the linear system
\begin{equation}\label{step:2}
	\langle x \ |\ u_i \rangle - \eta_i = \sum_{k\in I_m} {\nu}_k \langle u_k \ |\ u_i \rangle,\quad i \in I_m
\end{equation}
with respect to ${\nu}_k$, $k\in I_m$. If there exists $k\in I_m$ such that ${\nu}_k\leq 0$ let $m:=m+1$ and go to Step 1.\par
Step 3. Check if the following formula is satisfied
\begin{equation}\label{step:3}
	\forall i^\prime \in I_m^\prime \quad \langle x-\sum_{k \in I_m} {\nu}_k u_k \ |\ u_{i^\prime} \rangle \leq 0. 
\end{equation}
If not let $m:=m+1$ and go to Step 1.\par 
Step 4. The projection of $x$ onto $C$ is given by formula
\begin{equation}\label{step:4}
	P_C(x)=x-\sum_{i \in I_m} {\nu}_i u_i.
\end{equation}

By Theorem \ref{prop:projection_conditions}, among all the subsets $I\subset \Delta$ for which ${\nu}_k>0$, $k \in I$ given by \eqref{step:2} there exists at least one for which \eqref{step:3} holds.\par 
The proposed algorithm is suitable for parallelization. The parallelized version of the algorithm can be organized as follows. % $2^{\text{rank}\, G}$
In Step 2 of the algorithm at most $2^{n}-1$ systems are solved of at most $n$ equations. In Step 3 for each solution of system from Step 2 we need to calculate at most $n-1$ scalar products. 
Let us observe that according to Theorem \ref{prop:projection_conditions} the representation \eqref{step:4} may not be unique.
%Let us note, that Step 2 may be computed at most $2^{\text{rank}\, G}$ times.
%$\text{rank}\, G$ does not depend on dimensionality of $x$ and depends only on the number $n$ of halfspaces.

%	\end{remark}

\section{On latticial cone}\label{section:latticial}
%\mbox{\,}\indent 
In \cite{nemeth_how_to_project} the authors proposed a finite algorithm for finding the projection onto 
a class of cones, called \textit{latticial cones}. In this section we compare our approach 
developed in section \ref{section:results}
%implied by Theorem \ref{prop:projection_conditions}
with the approach proposed in \cite{nemeth_how_to_project}, where %In contrast to our approach 
the main tool was the Moreau decomposition theorem.

Let $K\subset H$ be a cone.
\
The polar of $K$ is the set
\begin{equation*}
	K^\circ:= \{ x\in H \ |\ \langle x \ |\ y \rangle \leq 0 \quad \forall y \in K \}.
\end{equation*}
\
\begin{definition}
	Let $H=\mathbb{R}^n$. $K$ is called  \textit{latticial} if $K=\text{cone} \{b_1,b_2,\dots,b_n\}$, where $b_1,b_2,\dots,b_n\in \mathbb{R}^n$ are linearly independent and
	\begin{equation*}
		\text{cone}\{b_1,b_2,\dots,b_n\}:=\{ x\in H \ |\ x=\sum_{i \in N} \alpha_i b_i, \ \alpha_i \geq 0\ \text{for}\ \ i \in N %, \sum_{i=1}^{n}\alpha_i =1 
		\}.
	\end{equation*} 
	When $K=\text{cone}\{b_1,b_2,\dots,b_n\}$ we say that $K$ is generated by $b_1,b_2,\dots,b_n$.
\end{definition}
Any \textit{latticial cone} is closed and convex.
\begin{lemma} \label{lemma:lattical_polar} \cite{nemeth_how_to_project} Let $K\subset \mathbb{R}^n$ be a latticial cone generated by vectors $b_1,b_2,\dots,b_n$. The polar cone to $K$ can be represented as
	\begin{displaymath} %\label{lattical:polar}
		K^\circ = \{ \mu_1 u_1 + \mu_2 u_2 + \dots + \mu_n u_n \ | \ \mu_i\geq 0,\ i\in N  \},
	\end{displaymath}
	where $u_j$, $j\in N$, solves the system
	\begin{displaymath}
		\langle u_j \ |\ b_i \rangle = -\delta_{i,j},\quad i\in N,
	\end{displaymath}
\end{lemma}
Since $K$ is closed and convex, $K= \{ x \in \mathbb{R}^n \ |\ \langle y \ |\ x \rangle \leq 0\quad \forall y \in K^\circ  \}$. By Lemma \ref{lemma:lattical_polar}, 
\begin{equation}\label{formula:alternative_latticial}
	x \in K \quad \Leftrightarrow \quad \langle x \ | u_i \rangle \leq 0\quad \forall i \in N.
\end{equation}

\begin{corollary}\label{corollary:independence} \cite{nemeth_how_to_project}
	For each subset $I$ of indices $I\subset N$, the vectors $b_i$, $i\in I$, $u_j$, $j\in N\backslash I $ are linearly independent.
\end{corollary}

\begin{theorem} \cite{Moreau1965} \label{theorem:decomposition} (Moreau decomposition theorem) 
	Let $K\subset \mathbb{R}^n$ be a closed convex cone and $x\in \mathbb{R}^n$. The following statements are equivalent.
	\begin{enumerate}
		\item $x=y+z$, $y\in K$, $z\in K^\circ$ and $\langle y \ |\ z \rangle =0$,
		\item $y=P_K x$ and $z=P_{K^\circ} x$.
	\end{enumerate}
\end{theorem}

The following fact has been proved in \cite[Theorem 2]{nemeth_how_to_project}. Here we provide an alternative proof based on the tools developed in Section \ref{section:Projection}.
\begin{theorem}
	Let $H=\mathbb{R}^n$ and let $K$ be a latticial cone generated by vectors $b_1,b_2,\dots,b_n$ and $x\notin K$. For each subset of indices $I\subset N$, $x$ can be represented in the form
	\begin{equation}\label{decomposition:lattice}
		x=\sum_{i\in I^{\prime}} \alpha_i b_i + \sum_{j \in I} \beta_j u_j
	\end{equation}
	with $I^{\prime}:=N \backslash I$. Moreover, among the subsets $I\subset N$ of indices there exists exactly one (the case $I=\emptyset$ is not excluded, but we exclude the case $I=N$ since $x\notin K$) with the property that %for the coefficients
	in \eqref{decomposition:lattice} one has $\beta_j>0$ for $j \in I$ and $\alpha_i\geq 0$ for $i \in I^{\prime}$ and
	\begin{equation*}
		P_K(x)=\sum_{i\in I^{\prime}} \alpha_i b_i.
	\end{equation*}
\end{theorem}
\begin{proof}
	The representation \ref{decomposition:lattice}  follows from Corollary \ref{corollary:independence}. To see the second assertion note that, by Theorem \ref{theorem:decomposition}, $x=a+b$, where $a\in K$, $b\in K^\circ$. Let $u_j$, $j\in N$ solve the system
	\begin{displaymath}
		\langle u_j \ |\ b_i \rangle = -\delta_{i,j},\quad i\in N.
	\end{displaymath}
	By formula \eqref{formula:alternative_latticial},   cone $K$ can be represented as
	$K=\bigcap_{i\in N} C_i$, where $C_i:=\{ h \in \mathbb{R}^n \ |\ \langle h \ |\ u_i \rangle \leq 0 \}$ for $i\in N$. By Theorem \ref{prop:projection_conditions}, there exists a set $I\subset N$ such that $\nu_i$, $i\in I$, given by formula \eqref{coeff:plus} are positive and for $i^\prime \in I^{\prime}:=N\backslash I$, the coefficients $\nu_{i^\prime}$ given by formula \eqref{coeff:notplus}
	are nonpositive and $P_K(x)=x-\sum_{i\in I}\tilde{\nu}_i u_i$, where $\tilde{\nu_i}=\nu_i/\text{det} G_{I,I}$ for $i\in I$.
	Since $\bar{x}=P_K(x)\in K$,
	\begin{equation}\label{projection:lattice}
		\bar{x}=\sum_{k\in N} \alpha_k b_k, \quad \alpha_k\geq 0\ \text{for}\ k\in N .
	\end{equation}
	Due to the linear independence of $b_k$, $k\in N$ and uniqueness of the projection, there exists exactly one system of coefficients $\alpha_k$, $k\in N$ such that $\bar{x}$ given by \eqref{projection:lattice} is the projection of $x$ onto $K$. Let $i \in I$. Taking the scalar product with vector $u_i$ at both sides of \eqref{projection:lattice} we obtain
	\begin{equation*}
		\langle \bar{x} \ |\ u_i \rangle = \sum_{k\in N} \alpha_k \langle b_k \ |\ u_i \rangle\quad  \Leftrightarrow \quad \langle \bar{x} \ |\ u_i \rangle = -\alpha_i.
	\end{equation*}
	From \eqref{projection:formula} and \eqref{system:solution} we have $\langle \bar{x} \ |\ u_i \rangle=0$ for $i \in I$. %complementarity slackness  conditions
	Thus $\alpha_i =0 $ for any $i\in I$. Hence,  \eqref{projection:lattice} reduces to
	\begin{displaymath} %\label{projection:lattice2}
		\bar{x}=\sum_{k\in I^{\prime}} \alpha_k b_k, \quad \alpha_k\geq 0\ \text{for}\ k\in I^{\prime} .
	\end{displaymath}
	Thus, by Theorem \ref{theorem:decomposition} and Lemma \ref{lemma:lattical_polar}, $x$ can be represented as
	\begin{equation}\label{representation:projection_lattice}
		x=\sum_{i\in I^{\prime}} \alpha_i b_i + \sum_{j \in N} \beta_j u_j \quad \alpha_i\geq 0\ \text{for}\ i \in I^{\prime} \ \text{and}\ \beta_j\geq 0 \ \text{for}\ j\in N, 
	\end{equation}
	where $P_K(x)=\sum_{i\in I^{\prime}} \alpha_i b_i$, $P_{K^\circ}(x)=\sum_{j \in N} \beta_j u_j$ and $K^\circ$ is generated by vectors $\{u_1,u_2,\dots,u_n\}$.
	Due to the linear independence of vectors $u_1,u_2,\dots,u_n$, the representation  $P_{K^\circ}(x)=\sum_{j \in N} \beta_j u_j$ is unique.
	
	Formula \eqref{representation:projection_lattice} can be rewritten as
	\begin{equation}\label{representation:projection_lattice2}
		P_K(x) = x - \sum_{j \in N} \beta_i u_i.
	\end{equation}
	Since the representation \eqref{representation:projection_lattice2} is unique and vectors $u_i$, $i\in N$ are orthogonal, by Theorem,  \ref{prop:projection_conditions} we obtain
	\begin{equation}\label{formula:connection}
		\beta_k=\left\{ \begin{array}{ll}
			\nu_k/\det G_{I,I} & \text{if}\ k\in I,\\
			0 & \text{if}\ k \notin I.
		\end{array}  \right.
	\end{equation}
	Thus, \eqref{representation:projection_lattice} can be written as
	\begin{equation*}
		x=\sum_{i\in I^{\prime}} \alpha_i b_i + \sum_{j \in I} \beta_j u_j \quad \alpha_i\geq 0\ \text{for}\ i \in I^{\prime} \ \text{and}\ \beta_j> 0 \ \text{for}\ j\in I.
	\end{equation*}
	Since the representation is unique the proof is completed.	
\end{proof}

%\begin{remark}
In case of projections onto latticial cones in $\mathbb{R}^n$, the algorithm proposed in section \ref{section:results}
%Section \ref{section:results}
differs from algorithm proposed in Section 3 of \cite{nemeth_how_to_project}. The differences follows from the fact that algorithm proposed in \cite{nemeth_how_to_project} is based on the Moreau decomposition theorem, whereas our algorithm is based on the Kuhn-Tucker conditions for the corresponding convex optimization problem. Formula \eqref{formula:connection} shows the relationship between the two algorithms. Namely, $\beta_k=\tilde{\nu}_k$ for $k \in I$, where $\tilde{\nu}_k$ is given as in the proof of Theorem \ref{prop:projection_conditions}, but the formula \eqref{representation:projection_lattice} is operational only in the finite-dimensional case for vectors $\{u_1,\dots,u_n\}$ which are linearly independent. Moreover, the computational cost of algorithm proposed in \cite{nemeth_how_to_project} depends strongly on the dimensionality of $x$.

\section{The case of Banach spaces}\label{section:Banach}

Let $(B,\ \|\cdot \|)$ be a Banach space. Let $C=\bigcap_{i\in N} C_i$, where $C_i=\{ h \in B \ |\ \langle f_i \ |\ h \rangle \leq \eta_i   \}$,  $f_i\in B^*\backslash\{0\}$, $\eta_i\in \mathbb{R}$ and $\langle \cdot \ |\ \cdot \rangle $ denotes the duality mapping.

Finding the projection of $x$ onto C is equivalent to solving the optimization problem
\begin{equation}\label{prob:Banach}
	\min_{h \in C } \ \frac{1}{r}\|h-x\|^r,\quad r\geq 1.
\end{equation}
For this problem the following Pshenichnyi-Rockafellar optimality conditions  hold (see e.g. Theorem 2.9.1 of \cite{convex_analysis}).
\begin{theorem}  A point	$\bar{x}\in B$ solves
	\eqref{prob:Banach} if and only if $\partial_{\frac{1}{r}\|\cdot -x\|^r} (\bar{x})\cap (-N(C,\bar{x}))\neq \emptyset$, where $N(C,\bar{x})$ is the normal cone to $C$ at $\bar{x}$ and $N(C,\bar{x}):=\{ h^* \in B^* \ |\ \forall h\in C \quad \langle h^* \ |\ h-\bar{x} \rangle \leq 0     \}$.
\end{theorem}

In the case of strictly convex reflexive Banach space $X$ every closed convex set $D$ is Chebyshev, i.e. for each $x\in X$ there exists a \textit{unique} point $P_D(x)\in D$ such that $\|x-P_D(x)\|=\inf \{ \| x-d\|,\ d\in D\}$.

Let $I(x):=\{i \in N \ |\ \langle f_i \ |\  x\rangle =\eta_i  \}$. We have $N(C,x)=cone(\{f_i, i\in I(x) \})$. When $k(\cdot):=\frac{1}{r}\|\cdot-x\|^r$ is Gateaux differentiable on $B$ we have $\bar{x}=P_C(x)$ if, and only if,
\begin{align}
	& 0\in k^\prime(\bar{x})+N(C,\bar{x})\notag\\
	& \iff \exists\ \nu_{i}\geq 0,\ i\in I(\bar{x})\quad k^\prime(\bar{x})=-\sum_{i\in I(\bar{x})}  \nu_i f_i\notag\\
	& \iff \exists\ \nu_{i}\geq 0,\ i\in I(\bar{x})\ \forall y\in B \quad \langle  k^\prime(\bar{x}) \ |\ y \rangle =  -\sum_{i\in I(\bar{x})} \nu_i \langle f_i \ |\ y \rangle\notag\\
	%& \iff \exists\ \nu_{i}\geq 0,\ i\in I(\bar{x})\ \forall y\in B \quad \langle  \bar{x}-x \ |\ y \rangle_+ =  -\sum_{i\in I(\bar{x})} \nu_i \langle f_i \ |\ y \rangle\\
	& \iff \exists\ \nu_{i}\geq 0,\ i\in I(\bar{x})\ \forall y\in B \quad k^\prime( \bar{x} , y ) =  -\sum_{i\in I(\bar{x})} \nu_i \langle f_i \ |\ y \rangle.\label{formula:projection_Banach_general}
	%&k^\prime(\bar{x}-x)=-\sum_{i\in N} \nu_i f_i\\
	%&\langle \bar{x}-x \ |\ f_j \rangle _+= \langle k^\prime(\bar{x}-x) \ |\ f_j \rangle  = -\sum_{i\in N} \nu_i \langle f_i \ |\ f_j \rangle \quad \forall_{j\in N}
\end{align}

In Banach spaces $\ell_p$, $p>1$  of all sequences $f=\{f_1,f_2,\dots,\}$ such that
\begin{displaymath}
	\|f\|_{\ell_p}:=\left(\sum_{i=1}^{+\infty} \|f_i\|^p\right)^\frac{1}{p}<+\infty
\end{displaymath}
it was shown in \cite[Example 8.1]{james1947} that the directional derivative of $k(\cdot)=\frac{1}{p}\|\cdot-x\|_{\ell_p}^p$ at $u$ in direction $v$ is given by formula
\begin{equation}\label{directionalderivative_lp}
	k^\prime(u,v)= \sum_{i=1}^{+\infty} |u_i-x_i|^{p-2}(u_i-x_i) v_i.
\end{equation}

In Banach spaces $L_p(\Omega)$, $p\geq 1$ of all functions $f:\ L_p(\Omega)\rightarrow \mathbb{R}$ such that
\begin{displaymath}
	\|f\|_{L^p}:=\left(\int_{t\in \Omega} |f(t)|^p\, \mu(dt)\right)^{\frac{1}{p}}<+\infty
\end{displaymath}
it was shown in \cite[Example 8.2]{james1947}, \cite[Example 13.12]{applied_analysis} that the directional derivative of $k(\cdot)=\frac{1}{p}\|\cdot-x\|_{\ell_p}^p$ at $u$ in direction $v$ is given by formula
\begin{equation}\label{directionalderivative_Lp}
	k^\prime(u,v)= \int_{\Omega} |u(t)-x(t)|^{p-2}(u(t)-x(t)) v(t) \, \mu(dt).
\end{equation}

%Let $n\in \mathbb{N}$ and $I\subset\{1,\dots,n\}$. Let $W$ be a quadratic matrix of size $|I|\times |I|$ such that for any $S\in I$, $S\neq \emptyset$, $\det W_{S,S}\neq 0$.
%Let $n\in \mathbb{N}$ and $N=\{1,\dots,n\}$. %and $I\subset \{1,\dots,n\}$, $I \neq \emptyset$. 

We start by discussing formulas for projections in spaces $\ell_p$, $p>1$. 
%Let $\lambda_j^i\in \mathbb{R}$ for any $i,j \in \mathbb{N}$.
For any matrix $L=(\lambda_i^j)$, $\lambda_i^j\in \mathbb{R}$, $i,j \in N$ and any finite subsets $A_1,A_2\subset N$ let $L_{A_1,A_2}$ be the matrix with  entries $\lambda_j^i$, where $i \in A_1$, $j\in A_2$. Let $e_i=(0,\dots,0,1,0\dots)$, with $1$ at the position $i$ denote the standard basis of $\ell_p$.

%	 For any $i,j\in N$ let $\lambda_j^i$ denote the element of $L$ in $i$-th row and $j$-th column.
%$s_I(i)$ row and $s_I(j)$ column.  
\begin{proposition}\label{prop:projection_lp}
	In the space $\ell_p$, $p>1$ consider $C_i=\{ h \in  \ell_p \ |\ \langle f_i \ |\ h \rangle \leq \eta_i \}$, where $\eta_i\in \mathbb{R}$, $f_i=\sum_{j\in \mathbb{N}} \lambda_{j}^i e_j \in \ell_{q}\backslash\{0\}$,  $i \in N$. Let $W:=\{ k \in \mathbb{N} \ |\ \forall i \in N \  \lambda_k^i=0   \}$.
	Then $\bar{x}$ is the projection of $x\in \ell_p\backslash C$ onto $C$ if and only if there exists $I\subset I(\bar{x})$ such that $\det L_{I,I} \neq 0$, and for $i \in I$
	\begin{displaymath}
		\begin{array}{l}
			\bar{x}_i=\dfrac{1}{\det L_{I,I}} \sum\limits_{k\in I} \tilde{\eta}_k B_{I}^k B_{I}^i \det L_{I\backslash \{k\} ,I\backslash \{i\} },  \\
			\dfrac{1}{\det L_{I,I}} \sum\limits_{k\in I} \xi_k B_{I}^k B_{I}^i \det L_{I\backslash \{k\} ,I\backslash \{i\} } >0, \\ 
		\end{array}
	\end{displaymath}
	for  $j \in \mathbb{N}\backslash(I\cup W)$
	\begin{equation}\label{cond:nonactive_index_lp}	|\bar{x}_j-x_j|^{p-2}{(\bar{x}_j-x_j)} =- \sum_{i\in I} \lambda_j^i \frac{1}{\det L_{I,I}}\left( \sum_{k\in I} \xi_k B_{I}^k B_{I}^i \det L_{I\backslash \{k\} ,I\backslash \{i\} }  \right)
	\end{equation}
	where $\tilde{\eta}_k=\eta_k-\sum_{j\in \mathbb{N}\backslash I} \lambda_j^k \bar{x}_j$,  $\xi_k=-|\bar{x}_k-x_k|^{p-2}(\bar{x}_k-x_k)$, $k \in I$,
	and
	\begin{displaymath}
		\bar{x}_k=x_k\quad \text{ for }\ k \in W.
	\end{displaymath}
	%	\begin{displaymath}
	%	\begin{array}{l}
	%	\det L_{I,I} \neq 0\\
	%	\forall i \in I\quad 	\bar{x}_i=\frac{1}{\det L_{I,I}} \sum\limits_{k\in I} \tilde{\eta}_k B_{I}^k B_{I}^i \det L_{I\backslash \{k\} ,I\backslash \{i\} }  \\
	%	\forall i \in I\quad \frac{1}{\det L_{I,I}} \sum\limits_{k\in I} \tilde{\xi}_k B_{I}^k B_{I}^i \det L_{I\backslash \{k\} ,I\backslash \{i\} } >0 \\ 
	%	\forall i \in \mathbb{N}\backslash(I\cup W)\  |\bar{x}_i-x_i|^{p-2}{(\bar{x}_i-x_i)} =- \sum\limits_{j\in I} \lambda_j^i \frac{1}{\det L_{I,I}}\left( \sum\limits_{k\in I} \tilde{\xi}_k B_{I}^k B_{I}^i \det L_{I\backslash \{k\} ,I\backslash \{i\} }  \right)\\
	%	\forall j \in W \quad \bar{x}_j=x_j,
	%	\end{array}
	%	\end{displaymath}
	
\end{proposition}	
\begin{proof}
	%Let $\bar{x}$ be the projection of $x$ onto $C=\bigcap_{i\in N} C_i$ and  $I(\bar{x})=\{i \in N \ |\ \langle f_i \ |\ \bar{x}\rangle =\eta_i \}$.
	By \eqref{formula:projection_Banach_general} and \eqref{directionalderivative_lp}, $\bar{x}\in C$ is the projection of $x\in \ell_p \backslash C$ onto $C=\bigcap_{i\in N} C_i$ if and only if
	\begin{equation}\label{formula:projection_Banach_lp}
		\exists\ \{\nu_{j}\}_{j\in  I(\bar{x})}\geq 0\ \forall y\in \ell_p \quad \sum_{k=1}^{+\infty} |\bar{x}_k-x_k|^{p-2}{(\bar{x}_k-x_k)} y_k=-\sum_{i\in I(\bar{x})} \nu_i \langle f_i \ |\ y \rangle,
	\end{equation}
	where $I(\bar{x})=\{i \in N \ |\ \langle f_i \ |\ \bar{x}\rangle =\eta_i \}$.
	%  	Let $i \notin I(\bar{x})$ and $y^i:=e_i=(0,\dots,0,1,0\dots)$. Then by \eqref{formula:projection_Banach_lp}
	%  	\begin{equation}\label{projection:Banach_lp_notactive}
	%  	\frac{|\bar{x}_i-x_i|^{p-2}{(\bar{x}_i-x_i)}}{\|\bar{x}-x\|_{\ell_p}^{p-2}}=0 \quad \iff \quad \bar{x}_i=x_i.
	%  	\end{equation}
	% 	Let $i \in I(\bar{x})$ and $y^i:=e_i=(0,\dots,0,1,0\dots)$. For any $a,b\in I$ let $\tilde{\lambda}_a^b:=\lambda_a^b$ if $a\in J_b$ and $0$ otherwise.  Then by \eqref{formula:projection_Banach_lp}
	% 	\begin{displaymath}
	% 	\frac{|\bar{x}_i-x_i|^{p-2}{(\bar{x}_i-x_i)}}{\|\bar{x}-x\|_{\ell_p}^{p-2}}=- \sum_{j\in I(\bar{x})} \tilde{\lambda}_i^j\nu_j.
	% 	\end{displaymath}
	Formula \eqref{formula:projection_Banach_lp} is equivalent to the following two conditions 
	\begin{align}
		\exists\ \{\nu_{i}\}_{i\in  I(\bar{x})}\geq 0\ \forall_{k\in \mathbb{N}\backslash W} \quad   &|\bar{x}_k-x_k|^{p-2}{(\bar{x}_k-x_k)} =- \sum_{i\in I(\bar{x})} \lambda_k^i\nu_i,\label{formula:proj:l_p}\\
		\forall_{k\in W} \quad   &|\bar{x}_k-x_k|^{p-2}{(\bar{x}_k-x_k)} =0.\label{formula:proj:l_p_else}
	\end{align}
	
	These conditions are obtained by taking $y=e_k=(0,\dots,0,1,0,\dots)$, $k\in \mathbb{N}$, in \eqref{formula:projection_Banach_lp}.
	%	\begin{equation}\label{formula:proj:l_p}
	%	\exists\ \{\nu_{j}\}_{j\in  I(\bar{x})}\geq 0\ \forall_{k\in \mathbb{N}\backslash W} \quad   |\bar{x}_k-x_k|^{p-2}{(\bar{x}_k-x_k)} =- \sum_{i\in I(\bar{x})} \lambda_i^k\nu_i
	%	\end{equation}
	%	and
	%	\begin{equation}\label{formula:proj:l_p_else}
	%	\forall_{i\in W} \quad   |\bar{x}_i-x_i|^{p-2}{(\bar{x}_i-x_i)} =0.
	%	\end{equation}
	Hence for all $k\in W$, $\bar{x}_k=x_k$. For any $i\in I(\bar{x})$ we have
	\begin{equation}\label{projection:lp_active_indexes}
		\langle f_i \ |\ \bar{x} \rangle = \eta_i \quad \iff \quad \sum_{j\in \mathbb{N}} \lambda_j^i \bar{x}_j = \eta_i .
	\end{equation}
	%	Let 
	%	\begin{displaymath}
	%		L=[\lambda_j^i]_{(i,j)\in I(\bar{x})^2}.
	%	\end{displaymath}
	If  $\det L_{I(\bar{x}),I(\bar{x})} =0$, by applying Lemma \ref{lemma:tech} to vectors $u_i:=[\lambda_k^i]_{k\in I(\bar{x})}$, $i\in I(\bar{x})$, we obtain the existence of an index set $ I\subset I(\bar{x})$, $I\neq \emptyset$ such that  $\det L_{I,I}\neq 0$ and
	\begin{equation}\label{cond:independent_system_lp}
		\forall i\in I(\bar{x}) \ \exists \tilde{\nu}_k>0,\ k \in I\quad \sum_{k \in I(\bar{x})} \lambda_k^i \nu_k =\sum_{k \in I} \lambda_k^i \tilde{\nu}_k.
	\end{equation}
	If $\det L_{I(\bar{x}),I(\bar{x})} \neq 0$ put $I:=I(\bar{x})$. Let $\tilde{\eta}_i:=\eta_i-\sum_{j\in \mathbb{N}\backslash I} \lambda_j^i \bar{x}_j$, $i\in I$. By \eqref{projection:lp_active_indexes}, for any $i\in I$
	\begin{displaymath}
		\sum_{k\in \mathbb{N}} \lambda_k^i \bar{x}_k = \eta_i\quad \iff \quad \sum_{k\in I} \lambda_k^i \bar{x}_k= \eta_i-\sum_{j\in \mathbb{N}\backslash I} \lambda_j^i \bar{x}_j,
	\end{displaymath}
	hence 
	\begin{displaymath}%\label{projection:lp_active_indexes_system}
		L_{I,I} [\bar{x}_i]_{i\in I}=[\tilde{\eta}_i]_{i\in I},
	\end{displaymath}
	%	For any $J\subset \mathbb{N}$, $J\neq \emptyset$ let $s_J(s):=\{ b \in J \ |\ b\leq s  \}$ and
	%	\begin{equation*}
	%		B_J^a:=\left\{\begin{array}{lcl}
	%			(-1)^{|s_J(a)|} & \text{if} & a \in J,  \\
	%			(-1)^{|J|+1} & \text{if} & a \notin J.
	%		\end{array}\right.
	%	\end{equation*} 
	%Then, by \eqref{projection:lp_active_indexes_system}, 
	and consequently
	\begin{displaymath}
		\bar{x}_i=\frac{1}{\det L_{I,I}} \sum_{k\in I(\bar{x})} \tilde{\eta}_k B_{I}^k B_{I}^i \det L_{I\backslash \{k\} ,I\backslash \{i\} },  \quad i\in I.
	\end{displaymath}
	Let $\xi_k:=-|\bar{x}_k-x_k|^{p-2}(\bar{x}_k-x_k)$, $k\in I$. By \eqref{formula:proj:l_p} and \eqref{cond:independent_system_lp},
	\begin{displaymath}
		\exists \tilde{\nu}_i>0,\ i\in I\ \forall k\in I \quad \sum_{i \in I} \lambda_k^i \tilde{\nu}_i = \xi_k.
	\end{displaymath}
	Hence
	\begin{displaymath}%\label{projection:lp_active_indexes_system_nu}
		\exists \tilde{\nu}_i>0,\ i\in I \quad L_{I,I} [\tilde{\nu}_i]_{i\in I}=[\xi_i]_{i\in I},
	\end{displaymath}
	%Then, by \eqref{projection:lp_active_indexes_system_nu},
	and consequently
	\begin{displaymath}
		\tilde{\nu}_i=\frac{1}{\det L_{I,I}} \sum_{k\in I} \xi_k B_{I}^k B_{I}^i \det L_{I\backslash \{k\} ,I\backslash \{i\} }, \quad i\in I.
	\end{displaymath}
	Thus, for all $j\in \mathbb{N}\backslash(I\cup W)$
	\begin{displaymath}
		|\bar{x}_j-x_j|^{p-2}{(\bar{x}_j-x_j)} =- \sum_{i\in I} \lambda_j^i \frac{1}{\det L_{I,I}}\left( \sum_{k\in I} \xi_k B_{I}^k B_{I}^i \det L_{I\backslash \{k\} ,I\backslash \{i\} }  \right),
	\end{displaymath}
	which completes the proof.
	%	Let us note that $L_{I,I}=L_{I,I}=L_{I,I}$.
	%	
	%	We obtain $\bar{x}$ is a projection of $x$ onto $C$ if only only if exists $I\subset I(\bar{x})$ such that
	%	\begin{displaymath}
	%	\begin{array}{l}
	%	\det L_{I,I}\neq 0,\\
	%	\forall i \in I\quad 	\bar{x}_i=\frac{1}{\det L_{I,I}}( \sum\limits_{k\in I} \tilde{\eta}_k B_{I}^k B_{I}^i \det L_{I\backslash \{k\} ,I\backslash \{i\} }  ),\\
	%	\forall i \in I\quad \frac{1}{\det L_{I,I}} \sum\limits_{k\in I} \tilde{\xi}_k B_{I}^k B_{I}^i \det L_{I\backslash \{k\} ,I\backslash \{i\} } > 0, \\ 
	%	\forall i \in \mathbb{N}\backslash(I\cup W)\ \xi_i =- \sum\limits_{j\in I} \lambda_j^i \frac{1}{\det L_{I,I}}\left( \sum\limits_{k\in I} \tilde{\xi}_k B_{I}^k B_{I}^i \det L_{I\backslash \{k\} ,I\backslash \{i\} }  \right),\\	
	%	\forall j \in W \quad \bar{x}_j=x_j,
	%	\end{array}
	%	\end{displaymath}
	%	where $\tilde{\eta}_k=\eta_k-\sum_{j\in N\backslash I} \lambda_j^k \bar{x}_j$, $k\in I$, $\xi_i=-|\bar{x}_i-x_i|^{p-2}(\bar{x}_i-x_i)$,  $i\in \mathbb{N}\backslash W$.
\end{proof}

\begin{remark}
	Let us note that if $\{m,m+1,\dots\}\subset W$, where $m\in \mathbb{N}$, then condition \eqref{cond:nonactive_index_lp} is actually required for $i\in \{1,\dots,m\}\backslash (I\cup W)$ only.
\end{remark}

\begin{example}\label{example:1}
	Recall that $N=\{1,\dots,n\}$. Let $J\subset N$, $J \neq \emptyset$. Let $|\delta_k|=1$ if $k\in J$ and $\delta_k=0$ otherwise. Consider $C_i=\{ h \in  \ell_p \ |\ \langle f_i \ |\ h \rangle \leq \eta_i \}$, where $f_i=\{0,\dots,0,\delta_i,0\dots\} \in \ell_{q}$, $\eta_i\in \mathbb{R}$, $i \in N$, i.e. $n$ is the highest index $i$ such that $\delta_i\neq 0$ appearing in the sets $C_i$ and $C_k=B$ for $k\in N\backslash J$. 
	
	%	Thus $\bar{x}\in C$ is a projection of $x\in \ell_p \backslash C$ onto $C=\bigcap_{i\in J} C_i$ if and only if
	%	\begin{equation}\label{formula:projection_Banach_lp2}
	%	\exists\ \nu_{i}\geq 0,\ i\in I(\bar{x})\ \forall y\in X \quad \sum_{i=1}^{+\infty} \frac{|\bar{x}_i-x_i|^{p-2}{(\bar{x}_i-x_i)} y_i}{\|\bar{x}-x\|_{\ell_p}^{p-2}}=-\sum_{i\in I(\bar{x})} \nu_i \langle e_i \ |\ y \rangle,
	%	\end{equation}
	%	where $I(\bar{x}):=\{i \in J \ |\ \langle e_i \ |\ \bar{x}\rangle =\eta_i \}$.
	Let $\bar{x}$ be the projection of $x\in\ell_p\backslash C $ onto $C=\bigcap_{i\in N} C_i$, $C\neq\emptyset$. By \eqref{formula:proj:l_p}-\eqref{formula:proj:l_p_else} we obtain
	\begin{align*}
		\exists\ \{\nu_{i}\}_{i\in  I(\bar{x})}\geq 0\ \forall_{k\in I(\bar{x})} \quad   &|\bar{x}_k-x_k|^{p-2}{(\bar{x}_k-x_k)} =- \delta_k \nu_k\\
		\forall i \in J\backslash I(\bar{x}) \quad %&|\bar{x}_i-x_i|^{p-2}(\bar{x}_i-x_i)=0 \quad \iff\quad 
		& \bar{x}_i=x_i.
	\end{align*}
	
	%	
	%	Hence $\bar{x}\notin C$ is a projection of $x$ onto $C$ if and only if $\bar{x}_i=x_i$ for $i\notin I(\bar{x}) $ and by \eqref{formula:proj:l_p}
	%	\begin{displaymath}
	%	%&\exists\ \nu_{i}\geq 0,\ i\in I(\bar{x})\ \forall y\in X \quad \sum_{i\in I(\bar{x})} \frac{|\bar{x}_i-x_i|^{p-2}{(\bar{x}_i-x_i)} y_i}{\|\bar{x}-x\|_{\ell_p}^{p-2}}=-\sum_{i\in I(\bar{x})} \nu_i y_i\notag\\
	%	% \iff \quad 
	%	\exists\ \nu_{i}\geq 0\ \forall_{i\in I(\bar{x})} \quad   |\bar{x}_i-x_i|^{p-2}(\bar{x}_i-x_i) =- \nu_i\delta_i\label{formula:proj:l_p2}.
	%	\end{displaymath}
	%which is obtained from \eqref{formula:projection_Banach_lp} by taking $y^i:=e_i=(0,\dots,1,\dots,0)$ and $i \in I(\bar{x})$.
	
	We will show that $\bar{z}$ given by
	\begin{equation}\label{projection:formula_z_lp}
		\bar{z}:=x-\sum_{k\in J} \xi_k e_k ,\text{where}\  
		\xi_k:=\left\{\begin{array}{ll}
			x_k-\delta_k\eta_k & \text{when}\ \delta_k x_k\geq \eta_k,\\
			0 & \text{otherwise}  .
		\end{array} \right.
		%	\xi_i=x_i-\eta_i\ \text{when}\ \langle e_i \ |\ x \rangle >\eta_i\ \text{and} 0 otherwise
	\end{equation}
	is a projection of $x$ onto $C$. For any $i\in I(\bar{x})$ we have
	\begin{displaymath}
		\langle f_i \ |\ \bar{x} \rangle = \eta_i \quad \iff \quad \delta_i \bar{x}_i = \eta_i \quad \iff \quad  \bar{x}_i = \delta_i\eta_i.
	\end{displaymath}	
	
	Consider the case $j\notin I(\bar{x})$. Then $\bar{x}_j=x_j$, and moreover, if $j\in J$, then $x_j=\bar{x}_j\leq \eta_i$ since $\bar{x}\in C_j$.
	%Then $\eta_i\neq \langle e_i \ |\ z \rangle =z_i = x_i-\xi_i$, hence $i\notin I(x)$ and $z_i=x_i$. Moreover, if $i\in J$, then $x_i=z_i\leq \eta_i$ since $z \in C_i$.
	
	Consider the case $i \in I(\bar{x})$ and $x_i>\delta_i\eta_i$. Then $\bar{x}_i=\delta_i\eta_i$ and by \eqref{formula:proj:l_p}
	\begin{displaymath}
		0>|\delta_i\eta_i-x_i|^{p-2}(\delta_i\eta_i-x_i)=-\delta_i\nu_i.
	\end{displaymath}
	Since $\nu_i\geq 0$, $\delta_i=1$ and $\delta_i x_i>\eta_i$.
	
	Consider the case $i \in I(\bar{x})$ and $ x_i\leq \delta_i\eta_i$. Then $\bar{x}_i=\delta_i\eta_i$ and by \eqref{formula:proj:l_p}
	\begin{displaymath}
		0\leq |\delta_i\eta_i-x_i|^{p-2}(\delta_i\eta_i-x_i)=-\delta_i\nu_i.
	\end{displaymath}
	Since $\nu_i\geq 0$ one of the following appears:
	\begin{itemize}
		\item $\delta_i=1$ and $x_i=\delta_i \eta_i=\bar{x}_i$,
		\item $\delta_i=-1$, $\delta_i x_i\geq \eta_i$.
	\end{itemize}	
	
	Thus $\bar{x}_i=\delta_i\eta_i$ when $\delta_i x_i\geq  \eta_i$, $i\in J$ and $\bar{x}_i=x_i$ otherwise, which proves \eqref{projection:formula_z_lp}.
\end{example}

\begin{remark}\label{example:2}
	Let $B=L^p(\Omega)$, $p\geq 1$. 
	%Let $I\subset \Omega$ such that $\mu(I)<+\infty$. 
	Let $f_i\in L^q(\Omega)$ %, $f_i(t)=0$ for $t\in \Omega \backslash W$, $i\in N$ 
	and $\eta_i\in \mathbb{R}$, $i\in N$.	
	Consider $C_i=\{ g \in L^p \ |\ \langle f_i \ |\ g \rangle \leq \eta_i \}$, $i\in \{1,\dots,n\}$. Let $W:=\{ t \in \Omega \ |\ \forall i\in N  \ f_i(t)=0 \}$. Then by \eqref{formula:projection_Banach_general} and \eqref{directionalderivative_Lp}, $\bar{x}$ is a projection of $x \in L^p\backslash C$ onto $C=\bigcap_{i\in N} C_i$ if any only if there exists $\nu_{i}\geq 0,\ i\in I(\bar{x})$ such that for all $y\in B$
	\begin{equation}\label{formula:projection_Banach_Lp}
		%\exists\ \nu_{i}\geq 0,\ i\in I(\bar{x})\ \forall y\in B \quad
		\int_{\Omega} |\bar{x}(t)-x(t)|^{p-2}(\bar{x}(t)-x(t)) y(t) \, \mu(dt)
		=  -\sum_{i\in I(\bar{x})} \nu_i 	\int_{\Omega} f_i(t) y(t) \, \mu(dt) ,
	\end{equation}
	where $I(\bar{x})=\{i\in N \ |\ 	\int_{\Omega} f_i(t)\bar{x}(t) \, \mu(dt)  = \eta_i  \}$.
	
	Taking onto account $y(t)=0$ for $t\in\Omega\backslash W$ and $y(t)=\bar{x}(t)-x(t)$ for $t\in W$ in \eqref{formula:projection_Banach_Lp} we obtain
	\begin{displaymath}
		%\exists\ \nu_{i}\geq 0,\ i\in I(\bar{x})\ \forall y\in B \quad
		\int_{W} |\bar{x}(t)-x(t)|^{p} \, \mu(dt)
		=  0 .
	\end{displaymath}
	Hence $\bar{x}(t)=x(t)$ for almost all $t\in W$. 
\end{remark}
\begin{remark}
	Let us note, that in nonreflexive spaces, the projection may not exist. In consequence, the conditions given in Remark \ref{example:2}, when applied to the space $L_1(\Omega)$ do not assure the existence of $\nu_{i}\geq 0,\ i\in I(\bar{x})$ such that for all $y\in B$ formula \eqref{formula:projection_Banach_Lp} holds. Proposition \ref{prop:projection_lp} concerns the spaces $\ell_p$, $p>1$ because in the space $l_1$ we do not have the formula for the directional derivative of the norm.
\end{remark}

\section{Conclusions}
%\mbox{\,}\indent 
%A characteristic feature of our approach is that we make an extensive use of the Kuhn-Tucker optimality conditions for convex optimization problems. 
The main advantage of our approach is 
%that the formulas are given in Hilbert space and 
that no requirement is needed  for any mutual relationships between vectors $u_i$, $i\in N$, which generate the halfspaces. In our approach the problem of finding projection reduces to the problem of finding Kuhn-Tucker multipliers $\nu_1,\dots,\nu_n$, the number of which coincides with the number of halfspaces and is independent of the dimensionality of $x$. The crucial point of the result is that we work in Hilbert spaces. We showed through examples that, in general, in Banach space, even if the norm is Gateaux-differentiable and $C$ is of particular form one cannot expect explicit formulas for projections onto $C$.

\bibliographystyle{plain}
\bibliography{references}
\end{document}